\documentclass[12pt]{amsart}

\usepackage{amsmath,amssymb,amsfonts, amsthm}

\renewcommand{\Re}{\text{Re}}
\renewcommand{\Im}{\text{Im}}

\title{Pretentiously detecting power cancellation}
\author{Junehyuk Jung}
\address{Princeton University, Department of Mathematics, Princeton, NJ 08544}
\email{junehyuk@math.princeton.edu}
\author{Robert J. Lemke Oliver}
\address{Emory University, Department of Mathematics and Computer Science, 400 Dowman Drive, Atlanta, GA 30322}
\email{rlemkeo@emory.edu}

\newtheorem{prop}{Proposition}
\newtheorem{theorem}{Theorem}
\newtheorem{lemma}{Lemma}
\newtheorem{question}{Question}
\theoremstyle{remark}\newtheorem*{remark*}{Remark}

\numberwithin{equation}{section}
\numberwithin{theorem}{section}
\numberwithin{prop}{section}
\numberwithin{lemma}{section}

\begin{document}

\begin{abstract}
Granville and Soundararajan have recently introduced the notion of \emph{pretentiousness} in the study of multiplicative functions of modulus bounded by $1$, essentially the idea that two functions which are similar in a precise sense should exhibit similar behavior.  It turns out, somewhat surprisingly, that this does not directly extend to detecting power cancellation - there are multiplicative functions which exhibit as much cancellation as possible in their partial sums that, modified slightly, give rise to functions which exhibit almost as little as possible.  We develop two new notions of pretentiousness under which power cancellation can be detected, one of which applies to a much broader class of multiplicative functions.
\end{abstract}

\maketitle

\section{Introduction and statement of results}

In a series of papers, Granville and Soundararajan (\cite{GS2001}, \cite{GS2003}, \cite{GS2007}, \cite{GS2008}, \cite{Granville2009} as a few examples) recently introduced the notion of \emph{pretentiousness} in the study of multiplicative functions taking values in the complex unit disc, essentially the idea that if two functions are ``close" in some sense, they should exhibit the same behavior.  One striking example of this philosophy is a theorem of Hal\'{a}sz \cite{Halasz1968}, which can be interpreted as saying that given a multiplicative function $f(n)$ with $|f(n)|\leq 1$ for all $n$, the partial sums
\[S_f(x):=\sum_{n\leq x} f(n)\]
are large if and only if $f(n)$ ``pretends" to be $n^{it}$ for some $t\in\mathbb{R}$ (possibly 0).  To make this precise, define the distance between two multiplicative functions $f(n)$ and $g(n)$ taking values in the complex unit disc to be
\[\mathbb{D}(f,g)^2:=\sum_{p}\frac{1-\Re(f(p)\bar{g}(p))}{p},\]
where here and throughout, the summation over $p$ is taken to be over primes.  This distance is typically infinite, but in the event that it is finite, we follow Granville and Soundararajan and say that $f(n)$ and $g(n)$ are pretentious to each other, or that $f(n)$ is $g(n)$-pretentious.  Hal\'{a}sz's theorem then says that if $S_f(x)\gg x$, then $f(n)$ must be $n^{it}$-pretentious for some $t$.  In other words, Hal\'{a}sz's theorem classifies those $f(n)$ for which $S_f(x)$ is as large as possible.  It is therefore natural to ask for which $f(n)$ we have that $S_f(x)$ is exceptionally small.  Since for generic $f(n)$ taking values in the complex unit disc, the best we can typically hope for is $S_f(x)\ll_\epsilon x^{1/2+\epsilon}$, we are interested in when $S_f(x)$ exhibits more than squareroot cancellation.  In particular, we ask the following question.

\begin{question}
\label{question:degree1}
If $f(n)$ is a completely multiplicative function, bounded by $1$ in absolute value, such that both $\sum_{n\leq x} |f(n)|^2 \gg x$ and $S_f(x)\ll x^{\frac{1}{2}-\delta}$ hold for some fixed $\delta>0$, must $f(n)$ be $\chi(n)n^{it}$-pretentious for some Dirichlet character $\chi$ and some $t\in\mathbb{R}$?
\end{question}

The reason for the condition that
\[\sum_{n\leq x} |f(n)|^2 \gg x\]
is twofold.  First, we wish to exclude functions like $f(n)=n^{-a}$ for some $a>0$, and second, this condition is necessary for $\mathbb{D}(f,f)$ to be finite, and therefore for $f(n)$ to be pretentious to any function.  In other words, this condition is necessary for $f(n)$ to fit into the context of pretentiousness.

To study Question \ref{question:degree1}, we first ask that if $f(n)$ is $\chi(n)$-pretentious for some character $\chi$, must $S_f(x)$ be small?  This turns out to not be the case - by taking $f(p)$ to be 1 for primes lying in one of a suitably sparse set of dyadic intervals and to be $\chi(p)$ otherwise, one obtains a function which is $\chi(n)$-pretentious, but for which $S_f(x)\gg x/\log x$ for infinitely many $x$.  We therefore must ask whether there is a stronger notion of pretentiousness which preserves power savings.  To that end, given $\beta\in(0,1]$, define the $\beta$-distance between $f(n)$ and $g(n)$ to be
\[\mathbb{D}_\beta(f,g)^2:=\sum_p \frac{1-\Re(f(p)\bar{g}(p))}{p^\beta},\]
and say that $f(n)$ and $g(n)$ are $\beta$-pretentious if $\mathbb{D}_\beta(f,g)$ is finite.  A natural guess would be that $\beta$-pretentiousness detects power savings, perhaps even down to $O(x^\beta)$.  Our first theorem shows that the first part of this guess is correct, but that the second is not.  To state it precisely, let $h(n)$ be the multiplicative function defined by the Dirichlet convolution 
\[g(n)=(f*h)(n)=\sum_{dm = n} f(d)h(m).\]  
From our notion of pretentiousness, it is natural to expect that if $f(n)$ and $g(n)$ are $\beta$-pretentious for some $\beta$, then $h(n)$ should be small in some sense (colloquially, we would say that $f(n)$ needs to be modified only slightly at the primes to obtain $g(n)$).  In practice, this smallness manifests itself as convergence properties of the Dirichlet series associated to $h(n)$.  In particular, we have the following, letting
\[H(\sigma):=\sum_{p\leq 4^{1/\sigma}} \sum_{k=0}^\infty\frac{|h(p^k)|^2}{p^{k\sigma}}.\]

\begin{theorem}
\label{thm:beta1}
Let $f(n)$ and $g(n)$ be multiplicative functions taking values in the complex unit disc such that $S_f(x)\ll x^\alpha$ for some $\alpha<1$ as $x\to\infty$, and suppose that $\mathbb{D}_\beta(f,g)<\infty$ for some $\beta\in (0,1]$.  
\begin{enumerate}
\item If $\sigma>3/4$ is such that $\sigma\geq \max(\alpha,(1+\beta)/2)$ and $H(2\sigma-1)$ is convergent, then $S_g(x)\ll x^\sigma$ as $x\to\infty$.  
\item If $f(n)$ and $g(n)$ are both completely multiplicative, then $S_g(x)\ll x^{\max(\alpha,(1+\beta)/2)}$.  Moreover, this bound is optimal in the following sense.  If $\beta \geq 2\alpha - 1$, there is a completely multiplicative function $f^\prime(n)$ that is $\beta$-pretentious to $f(n)$ and is such that $S_{f^\prime}(x)$ is not $O\left(x^{\frac{1+\beta}{2}-\epsilon}\right)$ for any $\epsilon>0$.
\end{enumerate}
\end{theorem}

\noindent \emph{Four remarks:}
\begin{enumerate}
\item The condition that the series $H(\sigma)$ is convergent is necessary to control the functions at small primes.  As an illustrative example, consider the multiplicative function $f(n)=(-1)^{n+1}$.  Since $f(n)$ differs from the constant function $\mathbf{1}$ only at the prime $2$, it is $\beta$-pretentious to $\mathbf{1}$ for all $\beta>0$.  We also observe that $S_f(x)=O(1)$, so Theorem \ref{thm:beta1} cannot apply to the pair $f, \mathbf{1}$.  The reason for this is that $h(2^k)=2^k$ for all $k\geq 0$, and so $H(\sigma)$ is divergent for any $\sigma\leq 1$.
\item While it's perhaps unsatisfying that $\beta$-pretentiousness only detects power savings down to $O\left(x^{\frac{1+\beta}{2}}\right)$ even for completely multiplicative functions, we note that the conclusion of the theorem may be able to be strengthened if $f(n)$ and $g(n)$ are assumed to be real-valued.  The reason for this is that our proof of optimality relies crucially on the fact that $1-\Re(f(p)\bar{g}(p))$ can be much smaller than $|f(p)-g(p)|$.
\item As the proof of Theorem \ref{thm:beta1} will show, if we have the stronger condition that $S_f(x)=o(x^\alpha)$, then we may conclude that $S_g(x)=o\left(x^\sigma\right)$.
\item There are quantitative versions of Hal\'{a}sz's theorem, due to Hal\'{a}sz \cite{Halasz1971}, Montgomery \cite{Montgomery1978}, Tenenbaum \cite{Tenenbaum1995}, and Granville and Soundararajan \cite{GS2003} and \cite{GSBook}, but all of these theorems are essentially unable to detect cancellation below $O\left(x\frac{\log\log x}{\log x}\right)$, and so are useless for the question of power cancellation.
\end{enumerate}

In view of Theorem \ref{thm:beta1}, which implies that $\beta$-pretentiousness is enough to detect power savings down to $O(x^{(1+\beta)/2})$, it's natural to ask what happens if $(1+\beta)/2<\alpha$, so that we can detect below the order of magnitude of $S_f(x)$.  That is, supposing we have precise information about $S_f(x)$, can we use $\beta$-pretentiousness to deduce precise information about $S_g(x)$?  This is the content of our second theorem.

\begin{theorem}
\label{thm:asymptotic}
Let $f(n)$ and $g(n)$ be multiplicative, with modulus bounded by 1, and suppose that $f(n)$ satisfies $S_f(x)=x^\alpha \xi(x)$ for some $\xi(x)$ satisfying $\xi(t)\ll_\epsilon t^\epsilon$ for all $\epsilon>0$, and where $\alpha>3/4$ if $f(n)$ is not completely multiplicative.  If $f(n)$ and $g(n)$ are $\beta$-pretentious for some $0<\beta<2\alpha-1$ and the series $H(2\sigma-1)$ is convergent for some $\sigma<\alpha$, then $S_g(x)=x^\alpha \tilde{\xi}(x)$ for an explicitly given $\tilde{\xi}(x)$ satisfying $\tilde{\xi}(t)\ll_\epsilon t^\epsilon$.  Moreover, if $\xi(t)$ satisfies the mean-square lower bound
\[\int_1^T |\xi(t)|^2 dt \gg_\epsilon T^{1-\epsilon},\]
then $\tilde{\xi}(t)$ does as well.  Here, we take $\gg_\epsilon$ to mean $\geq C_\epsilon T^{1-\epsilon}$ for all sufficiently large $T$, rather than merely not $o(T^{1-\epsilon})$.
\end{theorem}

We have in mind the following two applications of Theorem \ref{thm:asymptotic}:

\begin{enumerate}
\item If $S_f(x)$ satisfies an asymptotic formula, then so does $S_g(x)$.  For example, if the Dirichlet series associated to $f$, $L(s,f)$, has a finite number of poles on the line $\Re(s)=\alpha$ and is otherwise analytic on $\Re(s)>\alpha-\delta$ for some $\delta$, then standard Tauberian theorems (for example, see \cite{MV2007}) show that 
\[S_f(x) = \sum_{\begin{subarray}{c} \rho: \Re(\rho)=\alpha \\ \text{ord}_{s=\rho}L(s,f)<0 \end{subarray}} x^\rho P_\rho(\log x) + O(x^{\alpha-\delta+\epsilon}),\] 
where each $P_\rho(\log x)$ is a polynomial in $\log x$.  Thus, with the notation of Theorem \ref{thm:asymptotic}, we have that
\[\xi(x) = \sum_{\begin{subarray}{c} \rho:\Re(\rho)=\alpha \\ \text{ord}_{s=\rho}L(s,f)<0 \end{subarray}} x^{\Im(\rho)} P_\rho(\log x) + O(x^{-\delta+\epsilon}),\]
and it is easy to see that $\xi(x)$ satisfies the required upper bound.  Thus, we can apply Theorem \ref{thm:asymptotic}, and it turns out that in this application, $\tilde{\xi}(x)$ works out to be 
\[\tilde{\xi}(x) = \sum_{\begin{subarray}{c} \rho:\Re(\rho)=\alpha \\ \text{ord}_{s=\rho}L(s,f)<0 \end{subarray}} x^{\Im(\rho)} Q_\rho(\log x) + O(x^{-\delta^\prime})\]
for some suitably small $\delta^\prime >0$, where $Q_\rho(\log x)$ is a polynomial in $\log x$ of the same degree as $P_\rho(\log x)$.  Thus, the explicit nature of $\tilde{\xi}(t)$ is of use.
\item If $S_f(x)$ exhibits a consistent level of cancellation, then so does $S_g(x)$.  In the above situation, we made use of the explicit nature of $\tilde{\xi}(x)$ to deduce an asymptotic formula for $S_g(x)$, but in many cases, we would not be lucky enough to have an asymptotic formula for $S_f(x)$ with which to begin.  However, in many cases, it may be possible to deduce the weaker statement that $S_f(x)\not\ll_\epsilon x^{\alpha-\epsilon}$ for any $\epsilon>0$.  In this situation, the use of the mean-square lower bound becomes apparent - because $S_f(x)$ exhibits cancellation without satisfying an asymptotic formula, it is likely that $S_f(x)$ could be exceptionally small (perhaps even $0$) for some values of $x$, but it also seems that this occurence should be fairly rare.  We can therefore deduce from Theorem \ref{thm:asymptotic} that if $x^\alpha$ is the right order of magnitude of $S_f(x)$ in this sense, then $x^\alpha$ is also the right order of magnitude for $S_g(x)$.
\end{enumerate}

Somewhat unfortunately, Theorem \ref{thm:beta1} is unable to detect power cancellation below $O(x^{1/2})$, even if $f(n)$ and $g(n)$ are completely multiplicative, so it is unlikely to be of use in answering Question \ref{question:degree1}.  We therefore wish to develop yet another notion of distance that may be more suited to detecting more than squareroot cancellation.  In addition, we would like this new notion of distance to apply usefully to all multiplicative functions, rather than just completely multiplicative functions, which necessitates considering the values at higher prime powers, and we would also like to loosen the restrictions on the size of $f(n)$ and $g(n)$.  To that end, for any two multiplicative functions $f(n)$ and $g(n)$, not necessarily bounded by $1$ in absolute value, define the modified distance
\[\hat{\mathbb{D}}_{\beta,k}(f,g):=\sum_{p}\sum_{j=1}^k \frac{|f(p^j)-g(p^j)|}{p^{j\beta}},\]
and say that $f$ and $g$ are $(\beta,k)$-strongly pretentious if this distance is finite.  Analogous to what we did for Theorem \ref{thm:beta1}, define 
\[\hat{H}_Y(\sigma):=\sum_{p\leq Y} \sum_{k=1}^\infty \frac{|h(p^k)|}{p^{k\sigma}}.\]
\begin{theorem}
\label{thm:beta2}
Let $f(n)$ and $g(n)$ be multiplicative such that $S_f(x)\ll x^\alpha$, and suppose that $f(n)$ and $g(n)$ are $(\beta,k)$-strongly pretentious for some $(\beta,k)$.  Additionally suppose that $f(n),g(n)=o(n^\delta)$ for some $\delta>0$.  If $\sigma>1/(k+1)+\delta$ is such that $\sigma\geq \max(\alpha,\beta)$, then there is a $Y>0$ such if $H_Y(\sigma)$ is convergent, then $S_g(x)\ll x^\sigma$.
\end{theorem}

Let $\mathcal{S}_d$ denote the set of ``degree $d$" multiplicative functions, those $f(n)$ such that $f=f_1*\dots*f_d$, where each $f_i(n)$ is a completely multiplicative function of modulus bounded by $1$.  It is easy to see that if $f\in\mathcal{S}_d$, then $f(n)=o(n^\delta)$ for any $\delta>0$, and so it is possible to apply Theorem \ref{thm:beta2} to functions in $\mathcal{S}_d$.  Since the values $f(p^k)$ for $k\geq 1$ are completely determined by the set $\{f(p^i):1\leq i\leq d\}$ for any $f\in\mathcal{S}_d$, it is reasonable to expect that only the first $d$ prime powers should be relevant.  We are able to show this, and in fact, we are also able to show that $\hat{H}_Y(\sigma)$ is always convergent in the range $\sigma\geq\beta$ once $k$ is at least $d$.  That is, we obtain the following.

\begin{theorem}
\label{thm:degreed}
Let $f,g\in\mathcal{S}_d$ be such that $\hat{\mathbb{D}}_{\beta,d}(f,g)$ is finite for some $\beta>0$, and suppose that $S_f(x)\ll x^\alpha$.  Then $S_g(x)\ll x^{\max(\alpha,\beta)}$.
\end{theorem}

This paper is organized as follows.  In Section \ref{sec:beta}, we study the notion of $\beta$-pretentiousness, and prove Theorems \ref{thm:beta1} and \ref{thm:asymptotic}.  In Section \ref{sec:strong}, we prove Theorems \ref{thm:beta2} and \ref{thm:degreed} about strong pretentiousness.  We also engage in a discussion of the obstruction to extending these results.

\section{$\beta$-pretentiousness: Theorems \ref{thm:beta1} and \ref{thm:asymptotic}}
\label{sec:beta}

Initially, we fix only the following notation for this section: $f(n)$ and $g(n)$ are taken to be arbitrary multiplicative functions, and the function $h(n)$ is defined by the Dirichlet convolution $g(n)=(f*h)(n)$.  We further suppose that $S_f(x)\ll x^\alpha$ for some $\alpha$.  We begin, in Section \ref{subsec:thm:beta1}, by proving Theorem \ref{thm:beta1}, and conclude, in Section \ref{subsec:thm:asymptotic}, by establishing Theorem \ref{thm:asymptotic}.

\subsection{Proof of Theorem \ref{thm:beta1}}
\label{subsec:thm:beta1}

The key result which we use to exhibit cancellation in Theorem \ref{thm:beta1} is the following proposition.

\begin{prop}
\label{prop:cancellation}
Let $f(n)$,$g(n)$, and $h(n)$ be as above.  If the series
\[\sum_{n=1}^\infty \frac{|h(n)|^2}{n^\sigma}\]
is convergent for some $\sigma>0$, then $S_g(x)\ll x^{\max(\alpha,(1+\sigma)/2)}$.  Moreover, if $S_f(x)=o(x^\alpha)$, then $S_g(x)=o(x^{\max(\alpha,(1+\sigma)/2)})$.
\end{prop}

\begin{proof}
From the definition of $h(n)$, we have that
\begin{eqnarray*}
\sum_{n\leq x} g(n)
	&=& \sum_{m\leq x} h(m) \sum_{d\leq x/m} f(d) \\
	&\ll& x^\alpha \sum_{m\leq x} \frac{|h(m)|}{m^\alpha} \\
	&\leq& x^{\alpha} \left(\sum_{m=1}^\infty \frac{|h(m)|^2}{m^\sigma}\right)^{1/2}\left(\sum_{m\leq x}\frac{1}{m^{2\alpha-\sigma}}\right)^{1/2} \\
	&\ll& x^{\max(\alpha,(\sigma+1)/2)}.
\end{eqnarray*}
If we have the stronger assumption that $S_f(x)=o(x^\alpha)$, by splitting the sum over $m$ on the first line according to whether $m$ is large and proceeding in the same way, it is easily seen that $S_g(x)=o\left(x^{\max\left(\alpha,\frac{1+\sigma}{2}\right)}\right)$.  
\end{proof}

In light of Proposition \ref{prop:cancellation}, to prove the first part of Theorem \ref{thm:beta1}, it suffices to establish the following lemma.

\begin{lemma}
\label{lem:h1}
If $f(n)$, $g(n)$, and $h(n)$ are as above, $|f(n)|,|g(n)|\leq 1$ for all $n$, $f(n)$ and $g(n)$ are $\beta$-pretentious for some $\beta>0$, and $\sigma>1/2$ is such that $\sigma\geq\beta$, then the series
\[\sum_{n=1}^\infty \frac{|h(n)|^2 }{n^\sigma}\]
converges if the quantity
\[H(\sigma)=\sum_{p\leq 4^{1/\sigma}} \sum_{k=0}^\infty\frac{|h(p^k)|^2}{p^{k\sigma}}\]
is finite.
\end{lemma}
\begin{proof}
Since $|g(n)| \leq 1$ and $|f(n)| \leq 1$, we have that
\[
|h(p^k)| \leq 2^k
\]
for all $p$ and all $k$. Therefore for $p>3$, one has
\[
\sum_{k=0}^\infty\frac{|h(p^k)|^2}{p^{k\sigma}} \leq 1+ \frac{1-\mathrm{Re}(f(p)\overline{g(p)})}{p^\sigma}+ \frac{16}{p^{2\sigma}}(1-4/p^\sigma)^{-1}.
\]
Thus, our assumptions that $\sigma\geq\beta$ and
\[\mathbb{D}_\beta(f,g)=\sum_{p}\frac{1-\Re(f(p)\overline{g(p)})}{p^\beta}\]
is finite, together with the assumptions of the lemma, guarantee that the series
\[\sum_{n=1}^\infty \frac{|h(n)|^2}{n^\sigma} = \prod_p \left(\sum_{k=0}^\infty\frac{|h(p^k)|^2}{p^{k\sigma}}\right)\]
is absolutely convergent.
\end{proof}

To establish the cancellation for completely multiplicative functions claimed in the second part of Theorem \ref{thm:beta1}, we have the following lemma.

\begin{lemma}
\label{lem:h2}
If $f(n)$,$g(n)$, and $h(n)$ are as in Lemma \ref{lem:h1} and $f(n)$ and $g(n)$ are completely multiplicative, then the series \[\sum_{n=1}^\infty \frac{|h(n)|^2}{n^\beta}\] is convergent.
\end{lemma}  
\begin{proof}
Since $h(p^k)=g(p^{k-1})(g(p)-f(p))$ for all primes $p$ and all $k \geq 1$, we have that
\[
|h(p^k)|^2 \leq |g(p)-f(p)|^2 \leq 2(1-\Re(f(p)\bar{g}(p))).
\]
Therefore
\begin{align*}
\sum_{n=1}^\infty \frac{|h(n)|^2}{n^\beta} &=\prod_p \left(1+\sum_{k=1}^\infty \frac{|h(p^k)|^2}{p^{k\beta}}\right)\\
&\leq \prod_p \left(1+\frac{2(1-\Re(f(p)\bar{g}(p)))}{p^\beta}(1-p^{-\beta})^{-1}\right)\\
&\leq \exp\left(\sum_p \frac{2(1-\Re(f(p)\bar{g}(p)))}{p^\beta}(1-2^{-\beta})^{-1}\right)\\
&= \exp\left(2(1-2^{-\beta})^{-1} \mathbb{D}_\beta(f,g) \right) <\infty.
\end{align*}
\end{proof}

It is worth noting at this point that there is another natural approach to proving the theorem, albeit one that is not entirely within the bounds of the pretentious philosophy.  From the relation $g(n)=(f*h)(n)$, we have the Dirichlet series identity
\[L(s,g)=L(s,f)L(s,h).\]
The assumption that $S_f(x)\ll x^\alpha$ translates to $L(s,f)$ being analytic in the right half-plane $\Re(s)>\alpha$ and the assumption that $g(n)$ is $\beta$-pretentious to $f(n)$, in light of Lemma \ref{lem:h1} and the Cauchy-Schwarz inequality, implies that $L(s,h)$ is analytic in the region $\Re(s)>\max\left(3/4,\frac{1+\beta}{2}\right)$.  Standard arguments (e.g. Perron's formula) then imply the desired bound for $S_g(x)$.  Our proof of optimality will proceed along similar lines.  While it is somewhat unfortunate that we have to use this mildly non-pretentious argument, it is not entirely clear how to avoid its use.

\begin{lemma}
\label{lem:optimality}
Given any $\beta>0$ and a completely multiplicative function $f(n)$ of modulus bounded by $1$ such that $f(n)$ is $1$-pretentious to itself, there is a completely multiplicative function $g(n)$ that is $\beta$-pretentious to $f(n)$, and which does not satisfy $S_g(x)\ll x^{(1+\beta)/2-\epsilon}$ for any $\epsilon>0$.
\end{lemma}

\begin{proof}
First, we may assume that $L(s,f)$ is analytic in the region $\Re(s)>(1+\beta)/2-\delta$ for some $\delta>0$, otherwise we could simply take $g(n)$ to be $f(n)$.  Let 
\[
g(p):=e\left( \frac{\omega_p}{p^{\frac{1-\beta}{2}}\log\log p}\right) f(p),
\]
where $\omega_p=\pm 1$ is a system of signs to be specified later and, as is standard, $e(x):=e^{2\pi i x}$.  It is easy to verify that $g(n)$ is $\beta$-pretentious to $f(n)$.  We now compute the Euler product for $L(s,h)$ using the Taylor expansion of $e(x)$, getting that
\begin{eqnarray*}
L(s,h)
	&=& \prod_p \left( 1 + \frac{g(p)-f(p)}{p^s} + O\left(p^{-2s}\right)\right) \\
	&=& \prod_p \left(1 + \frac{2\pi i \omega_p f(p)}{p^{s+\frac{1-\beta}{2}}\log\log p}+O\left(p^{-2s}+p^{-s-1+\beta}\right)\right).
\end{eqnarray*}
The convergence of $L(s,h)$ at $s=\frac{1+\beta}{2}$ is thus dictated by the behavior of the series
\[P_f(\tau):=\sum_{p} \frac{i\omega_p f(p)}{p^\tau\log\log p}\]
as $\tau$ tends to 1 from the right.  In particular, $L(s,h)$ will have a singularity at $s=\frac{1+\beta}{2}$ if we can force either the real part of $P_f(\tau)$ to tend to infinity, accounting for a (possibly fractional order) pole, or, failing that, to have the real part of $P_f(\tau)$ converge but the imaginary part diverge to infinity, accounting for an essential singularity.  Obviously, we now choose $\omega_p$ to ensure one of these situations.  If the series
\[\sum_p \frac{\Im(f(p))}{p\log\log p}\]
is not absolutely convergent, we choose $\omega_p=-\text{sign}(\Im(f(p)))$, forcing $\Re(P_f(\tau))$ to diverge to infinity.  If the series is absolutely convergent, we choose $\omega_p=\text{sign}(\Re(f(p)))$, observing that
\begin{eqnarray*}
\sum_p \frac{|\Re(f(p))|}{p^\tau\log\log p} + \sum_p \frac{|\Im(f(p))|}{p^\tau\log\log p} 
	&\geq& \sum_p \frac{\Re(f(p))^2+\Im(f(p))^2}{p^\tau\log\log p} \\
	&=& \sum_p \frac{|f(p)|^2}{p^\tau\log\log p} \\
	&\geq& \sum_p \frac{1}{p^\tau\log\log p} - \mathbb{D}_1(f,f),
\end{eqnarray*}
which tends to infinity as $\tau\to 1^+$.  We thus have that
\[\Im\left(\sum_p \frac{i\omega_p f(p)}{p\log\log p}\right) = \sum_p \frac{|\Re(f(p))|}{p\log\log p} = \infty,\]
from which we conclude that $\Im(P_f(x))$ tends to infinity.  We have thus constructed $g(n)$ so that $L(s,h)$ has a singularity at $s=\frac{1+\beta}{2}$, so provided that $L\left(\frac{1+\beta}{2},f\right)\neq 0$, we obtain the result.  If $L\left(\frac{1+\beta}{2},f\right)=0$, there is a $t\in\mathbb{R}$ such that $L\left(\frac{1+\beta}{2}+it,f\right)\neq 0$.  We make the obvious modifications to the construction above to force $L(s,h)$ to have a singularity at $s=\frac{1+\beta}{2}+it$.
\end{proof}

\subsection{Proof of Theorem \ref{thm:asymptotic}}
\label{subsec:thm:asymptotic}

We now suppose we are in the situation of Theorem \ref{thm:asymptotic}.  That is, we assume that $f(n)$ is multiplicative, of modulus bounded by $1$, and is such that
\[S_f(x)=x^\alpha \xi(x)\]
for some function $\xi(x)$ satisfying $\xi(t)\ll_\epsilon t^\epsilon$ for all $\epsilon>0$, and we also assume that $\beta<2\alpha-1$.  In addition, if $f(n)$ is not completely multiplicative, we assume that $\alpha>3/4$ and that the series $H(2\sigma-1)$ is convergent.  To establish a similar formula for $S_g(x)$, we note that
\begin{eqnarray*}
\sum_{n\leq x} g(n) 
	&=& \sum_{m\leq x} h(m) \sum_{d\leq x/m} f(d) \\
	&=& x^\alpha \sum_{m\leq x} \frac{h(m)}{m^\alpha} \xi(x/m),
\end{eqnarray*}
and so we naturally define $\tilde{\xi}(x)$ to be the convolution
\[\tilde{\xi}(x):=\sum_{m\leq x} \frac{h(m)}{m^\alpha} \xi(x/m).\]
To see that $\tilde{\xi}(x)\ll x^\epsilon$, we merely note that
\[|\tilde{\xi}(x)| \leq \sum_{m\leq x} \frac{|h(m)|}{m^\alpha} |\xi(x/m)| \ll_\epsilon x^\epsilon \sum_{m\leq x} \frac{|h(m)|}{m^{\alpha+\epsilon}}.\]
Our assumptions guarantee that the series on the right is convergent, whence the claimed bound.  Now, suppose that 
\[\int_1^T |\xi(t)|^2 dt \gg_\epsilon T^{1-\epsilon}.\]
M\"{o}bius inversion gives that
\[\xi(x)=\sum_{m\leq x} \frac{\tilde{h}(m)}{m^\alpha} \tilde{\xi}(x/m),\]
where $\tilde{h}(n)$ is the Dirichlet inverse of $h(n)$ (i.e., $(h*\tilde{h})(1)=1$ and $(h*\tilde{h})(n)=0$ for $n>1$).  Using this and the Cauchy-Schwarz inequality in the above, we obtain that
\begin{eqnarray*}
T^{1-\epsilon}
	&\ll_\epsilon & \int_1^T \left(\sum_{m\leq t} \frac{|\tilde{h}(m)|^2}{m^\beta}\right)\left(\sum_{m\leq t}\frac{|\tilde{\xi}(t/m)|^2}{m^{2\alpha-\beta}}\right)dt \\
	&\leq& \sum_{m=1}^\infty \frac{|\tilde{h}(m)|^2}{m^\beta} \int_1^T \sum_{m\leq t}\frac{|\tilde{\xi}(t/m)|^2}{m^{2\alpha-\beta}} dt \\
	&=& \sum_{m=1}^\infty \frac{|\tilde{h}(m)|^2}{m^\beta} \sum_{m\leq T} \frac{1}{m^{2\alpha-\beta-1}}\int_1^{T/m} |\tilde{\xi}(t)|^2 dt \\
	&\ll& T^{2-2\alpha+\beta} \int_1^T |\tilde{\xi}(t)|^2\frac{dt}{t^{2-2\alpha+\beta}},
\end{eqnarray*}
and we note that the infinite series is convergent for the same reason as the series for $|h(n)|^2$, so we have absorbed it into the implied constant.  Now, let 
\[I:=\int_1^T |\tilde{\xi}(t)|^2 dt,\]
and apply H\"{o}lder's inequality to get that
\begin{eqnarray*}
\int_1^T |\tilde{\xi}(t)|^2\frac{dt}{t^{2-2\alpha+\beta}}
	&\leq& I^{\frac{2\alpha-\beta-1}{2}} \left( \int_1^T \frac{|\tilde{\xi}(t)|^2}{t^{\frac{2(2-2\alpha+\beta)}{3-2\alpha+\beta}}} dt\right)^{\frac{3-2\alpha+\beta}{2}} \\
	&\ll_\epsilon& I^{\frac{2\alpha-\beta-1}{2}} \left( \int_1^T t^{\frac{-2(2-2\alpha+\beta)}{3-2\alpha+\beta}+\epsilon} dt\right)^{\frac{3-2\alpha+\beta}{2}} \\
	&\ll& I^{\frac{2\alpha-\beta-1}{2}} T^{\frac{2\alpha-\beta-1}{2}+\epsilon}.
\end{eqnarray*}
Using this in the above, we obtain that
\[I^{\frac{2\alpha-\beta-1}{2}} T^{\frac{3-2\alpha+\beta}{2}+\epsilon} \gg_\epsilon T^{1-\epsilon},\]
and so we have that
\[I^{\frac{2\alpha-\beta-1}{2}} \gg_\epsilon T^{\frac{2\alpha-\beta-1}{2}-\epsilon},\]
and the result follows, concluding the proof of Theorem \ref{thm:asymptotic}.

Since the Dirichlet series $L(s,h)$ for $\Re(s)\geq \alpha$ plays a critical role in the definition of $\tilde{\xi}(x)$, it is useful to know whether it is 0.  In particular, in applying Theorem \ref{thm:asymptotic} in the case when $S_f(x)$ satisfies an asymptotic formula, we might potentially lose a term in our formula if $L(\rho,h)=0$ for some pole $\rho$ of $L(s,f)$.  However, we have the following simple observation.

\begin{lemma}
\label{lem:nonvanishing}
If $f(n)$ and $g(n)$ are completely multiplicative and as above, then the Dirichlet series $L(s,h)$ associated to $h(n)$ is non-zero in the region $\Re(s)>(1+\beta)/2$.
\end{lemma}
\begin{proof}
Since $h(n)$ is defined by the relation $g=f*h$, we have the Dirichlet series formula
\[L(s,h)=\frac{L(s,g)}{L(s,f)}.\]
By Lemma \ref{lem:h2}, this is absolutely convergent in the region $\Re(s)>(1+\beta)/2$.  If we define $\tilde{h}(n)$ by $f=g*\tilde{h}$, the same argument applies to $L(s,\tilde{h})$.  Since we also have that
\[L(s,\tilde{h})=\frac{1}{L(s,h)},\]
this immediately yields the result.
\end{proof}

Of course, if $f(n)$ and $g(n)$ are not completely multiplicative, the analogue of Lemma \ref{lem:nonvanishing} can still be obtained with Lemma \ref{lem:h1} replacing Lemma \ref{lem:h2}.

\section{Strong pretentiousness: Theorems \ref{thm:beta2} and \ref{thm:degreed}}
\label{sec:strong}

As in Section \ref{sec:beta}, we consider separately the proofs of Theorems \ref{thm:beta2} and \ref{thm:degreed}.  In Section \ref{subsec:squarefree}, we consider another interesting application of Theorem \ref{thm:beta2}, which illustrates some of the obstructions to extending the result.

\subsection{Proof of Theorem \ref{thm:beta2}}
\label{subsec:thm:beta2}

We now let $f(n)$, $g(n)$, and $h(n)$ be as in the hypotheses of Theorem \ref{thm:beta2}.  Thus, $f(n)$ and $g(n)$ are multiplicative and satisfy $f(n),g(n)=o(n^\delta)$ for some $\delta>0$, and $h(n)$ is defined by $g(n)=(f*h)(n)$.  We begin with the following analogue of Proposition \ref{prop:cancellation}, which we present without proof.

\begin{prop}
\label{prop:strongcancel}
Let $f(n)$, $g(n)$, and $h(n)$ be as above.  Suppose that $S_f(x)\ll x^\alpha$ for some $\alpha>0$.  If the series
\[\sum_{n=1}^\infty \frac{|h(n)|}{n^\sigma}\]
is convergent for some $\sigma>0$, then $S_g(x)\ll x^{\max(\alpha,\sigma)}$.
\end{prop}

As in the proof of Theorem \ref{thm:beta1}, we now turn to determining the convergence properties of the series
\[\sum_{n=1}^\infty \frac{|h(n)|}{n^\sigma}.\]
For any $k\geq 1$, define $D_f(k,p)$ to be the determinant of the $k \times k$ matrix $(a_{ij})$ given by
\[
a_{ij} = \left\{ \begin{array}{ll} f(p^{i-j+1}) & \text{if } i-j+1 \geq 0 \\ 0 & \text{otherwise.} \end{array}\right.
\]
From the definition of $h(n)$, we have that
\[
g(p^k) - f(p^k) = \sum_{j=1}^k f(p^{k-j})h(p^j)
\]
and, incorporating all the powers up to $n$, we may express this in terms of the $n\times n$ matrix $A:=\big(f(p^{i-j})\big)_{i,j\leq n}$ (where we put $f(p^j)=0$ if $j<0$), by
\[
A\cdot \left(h(p), \cdots , h(p^n)\right)^t = \left(g(p)-f(p),\cdots ,g(p^n)-f(p^n)\right)^t.
\]
Observing that $(-1)^kD_f(k,p)$ is the $(n,n-k)$-th entry of the matrix $A^{-1}$, we get that
\[
h(p^n) = \sum_{k=0}^{n-1} (-1)^k(g(p^{n-k})-f(p^{n-k}))D_f(k,p).
\]
Therefore for $\sigma>0$ sufficiently large, we have that
\begin{align*}
\sum_{n=1}^\infty \frac{|h(p^n)|}{p^{n\sigma}} &\leq \sum_{n=1}^\infty \left(\sum_{k=1}^n|f(p^k)-g(p^k)|\cdot|D_f(n-k,p)|\right)p^{-n\sigma}\\
&=\left(\sum_{n=0}^\infty \frac{|D_f(n,p)|}{p^{n\sigma}}\right)\left(\sum_{m=1}^\infty \frac{|f(p^m)-g(p^m)|}{p^{m\sigma}}\right).
\end{align*}
We handle each expression on the right hand side separately.
\begin{lemma}
\label{lem:detbound}
If $f(n) = o(n^\delta)$ and $\sigma>\delta$, then for all but finitely many $p$, the series
\[
\sum_{n=0}^\infty \frac{|D_f(n,p)|}{p^{n\sigma}}
\]
is convergent and uniformly bounded.
\end{lemma}
\begin{proof}
Let $M(k,p)$ be the maximum of the absolute value of the determinants of the $k \times k$ matrices $(a_{ij})$ which satisfy
\[
|a_{ij}| \leq \left\{ \begin{array}{ll} p^{(i-j+1)\delta} & \text{if } i-j+1 \geq 0 \\ 0 & \text{otherwise.} \end{array} \right.
\]
Then, we observe that
\[
M(k+1,p) \leq 2p^\delta M(k,p) 
\]
by cofactor expansion, and that $M(1,p)=p^\delta$. It therefore follows that
\[
M(k,p) \leq 2^{k-1}p^{k\delta},
\]
which implies that the bound
\[
|D_f(n,p)| < (2p^\delta)^n
\]
holds for all but finitely many $p$.
\end{proof}
\begin{lemma}
\label{lem:largep}
Assume that both $f(n) = o(n^\delta)$ and $g(n) = o(n^\delta)$. If $\sigma$ and $k$ are chosen such that $\sigma >1/k + \delta$, then for some sufficiently large $N$, the series
\[
\sum_{p>N} \sum_{n=k}^\infty \frac{|f(p^n)-g(p^n)|}{p^{n\sigma}}
\]
is convergent.
\end{lemma}
\begin{proof}
By assumption, we have that
\[
|f(p^n)-g(p^n)| \leq p^{n\delta}
\]
for all sufficiently large $p$.
\end{proof}

With these two lemmas in hand, we are now able to complete the proof of Theorem \ref{thm:beta2}.

\begin{proof}[Proof of Theorem \ref{thm:beta2}]
Let $f(n)$ and $g(n)$ be multiplicative such that $f(n)=o(n^\delta)$ and $g(n)=o(n^\delta)$, and such that $f(n)$ and $g(n)$ are $(\beta,k)$-strongly pretentious.  Moreover, suppose that $S_f(x)\ll x^\alpha$ and $\sigma>1/(k+1)+\delta$ is such that $\sigma\geq\max(\alpha,\beta)$.  By applying Lemmas \ref{lem:detbound} and \ref{lem:largep}, we can find $C,Y>0$ large enough such that the bound
\[
\sum_{n=0}^\infty \frac{|D_f(n,p)|}{p^{n\sigma}} < C
\]
holds for all $p>Y$ and the series
\[
\sum_{p>Y} \sum_{n=k+1}^\infty \frac{|f(p^n)-g(p^n)|}{p^{n\sigma}}
\]
is convergent. It follows that
\begin{align*}
\sum_{p>Y}\sum_{n=1}^\infty \frac{|h(p^n)|}{p^{n\sigma}} &\leq \sum_{p>Y}\left(\sum_{n=0}^\infty \frac{|D_f(n,p)|}{p^{n\sigma}}\sum_{m=1}^\infty \frac{|f(p^m)-g(p^m)|}{p^{m\sigma}}\right)\\
&< C\sum_{p>Y}\sum_{m=1}^\infty \frac{|f(p^m)-g(p^m)|}{p^{m\sigma}}\\
&\leq C\left(\hat{\mathbb{D}}_{\beta,k}(f,g)+\sum_{p>Y} \sum_{n=k+1}^\infty \frac{|f(p^n)-g(p^n)|}{p^{n\sigma}}\right)\\
&<\infty.
\end{align*}
Hence, if the series
\[
\sum_{p\leq Y}\sum_{n=1}^\infty \frac{|h(p^n)|}{p^{n\sigma}} 
\]
is convergent, we have that
\[
\sum_{n=1}^\infty \frac{|h(n)|}{n^\sigma} <\infty,
\] 
and by applying Proposition \ref{prop:strongcancel}, we deduce the theorem.
\end{proof}

\subsection{Proof of Theorem \ref{thm:degreed}}
\label{subsec:thm:degreed}
Suppose that $f(n)$ and $g(n)$ are multiplicative functions of degree $d$, and that $\hat{\mathbb{D}}_{\beta,d}(f,g)<\infty$.  To prove Theorem \ref{thm:degreed}, we must show that the series 
\[\sum_{n=1}^\infty \frac{|h(n)|}{n^\sigma}\]
is convergent for $\sigma\geq\beta$.  From the Euler product for $|h(n)|$, we have that the above series converges if and only if the series
\[\sum_p \sum_{n=1}^\infty\frac{|h(p^n)|}{p^{n\sigma}}\]
converges.  Now, recall for each prime $p$, that
\[\sum_{n=1}^\infty \frac{|h(p^n)|}{p^{n\sigma}} \leq \left(\sum_{n=0}^\infty \frac{|D_f(n,p)|}{p^{n\sigma}} \right)\left(\sum_{m=1}^\infty \frac{|f(p^m)-g(p^m)|}{p^{m\sigma}}\right),\]
where $D_f(n,p)$ is as in Section \ref{subsec:thm:beta2}.  We therefore break the proof into two lemmas, from which the theorem will immediately follow.

\begin{lemma}
\label{lem:degreeddet}
If $f(n)$ is a degree $d$ multiplicative function, then, for all $p$, the series
\[\sum_{n=0}^\infty \frac{|D_f(n,p)|}{p^{n\sigma}}\]
converges and is bounded independent of $p$.
\end{lemma}

\begin{proof}
Given any multiplicative function $g(n)$, we established in the discussion leading to Lemmas \ref{lem:detbound} and \ref{lem:largep}, that if $g(n)=(f*h)(n)$, that
\[h(p^n)=\sum_{k=0}^{n-1} (-1)^k(g(p^{n-k})-f(p^{n-k}))D_f(k,p),\]
which we think of as a linear polynomial in the variables $g(p^i)$ for $i=1,\dots,n$.  We note in particular that the coefficient of $g(p^{n-j})$ is $D_f(j,p)$ for all $j$.  On the other hand, from the definition of $h(n)$, we have the Euler product identity
\[\prod_p \left(\sum_{n=0}^\infty h(p^n)p^{-ns}\right) = \prod_p \left(\sum_{n=0}^\infty g(p^n)p^{-ns}\right)\left(1-f_1(p)p^{-s}\right)\dots\left(1-f_d(p)p^{-s}\right),\]
where the $f_i(n)$ are the constituent completely multiplicative functions of $f(n)$.  Thus, $h(p^n)$ can be expressed as a linear combination of the variables $g(p^i)$ for $i=n-d,\dots,n$.  Combining these two observations, we conclude that $D_f(k,p)=0$ for $k\geq d+1$.  The result follows by noting that each of the $D_f(k,p)$ for $k\leq d$ can be bounded independent of $p$.
\end{proof}

\begin{lemma}
\label{lem:degreeddist}
Let $f(n)$ and $g(n)$ be degree $d$ multiplicative functions, and suppose that $\hat{\mathbb{D}}_{\beta,d}(f,g)<\infty$.  Then $\hat{\mathbb{D}}_{\beta,\infty}(f,g)<\infty.$
\end{lemma}

\begin{proof}
We begin with some general notation.  For any given pair of integers $k,d\geq 0$, define the homogeneous symmetric polynomials of degree $k$ in $d$ variables $r_k^d$ and $q_k^d$ by
\[
r_k^d(x_1, \cdots, x_d) := \left\{ \begin{array}{ll} 1, & \text{if } k=0 \\\displaystyle  \sum_{1\leq i_1 < \cdots < i_k \leq d} x_{i_1} x_{i_2} \cdots x_{i_k}, & \text{if } 1 \leq k\leq d \\
0, & \text{if } k>d, \end{array} \right.
\]
 and
\[
q_k^d(x_1,\cdots, x_d) := \sum_{j_1+ \cdots +j_d=k} x_1^{j_1}x_2^{j_2} \cdots x_d^{j_d}.
\]
Then for an auxiliary variable $X$, we have that
\begin{align*}
\sum_{k=0}^\infty q_k^d X^k &= \prod_{j=1}^n \left( \sum_{k=0}^\infty x_j^kX^k\right) \\
&=\prod_{j=1}^n \left(1-x_jX\right)^{-1}\\
&=\left(\sum_{k=0}^d (-1)^kr_k^dX^k\right)^{-1}.
\end{align*}
This implies that
\[
\sum_{j=0}^k (-1)^j r_{k-j}^d q_j^d =0
\]
holds for $k \geq 1$. From this and using the fact that $\{r_k^d\}_{k=0,1,\cdots,d}$ are algebraically independent, we see that $\{q_k^d\}_{k=0,1,\cdots,d}$ are algebraically independent. Therefore we can find polynomials $p_1, \cdots, p_d$ such that
\[
p_k(q_1^d, \cdots, q_d^d) = r_k^d
\]
for any $0 \leq k \leq d$.

Now, if $f(n)$ is a multiplicative function of degree $d$, we set $\alpha_0(f,p)=1$ and put
\[
\alpha_k(f,p) := p_k(f(p), \cdots, f(p^d))
\]
for $k=1,\cdots, d$. Then
\[
\sum_{k=0}^d (-1)^k\alpha_k(f,p) f(p^{n+d-k})=0
\]
for any $p$ and $n\geq 0$. In particular, for any multiplicative functions $f(n)$ and $g(n)$ of degree $d$, we have that
\[
|\alpha_k(f,p)-\alpha_k(g,p)| \ll_d |f(p)-g(p)|+|f(p^2)-g(p^2)|+ \cdots +|f(p^d)-g(p^d)|
\]
for any $k=1,\cdots,d$ and $p$.  We are now ready to prove the lemma.  Assume that $n\geq d+1$. Observing that $f(p^n) \ll_d n^{d-1}$ and $\alpha_k(f,p) \ll_d 1$, we have
\begin{align*}
|f(p^n)-g(p^n)| = &\left|\sum_{k=1}^d (-1)^k\alpha_k(f,p) f(p^{n-k})+(-1)^k\alpha_k(g,p) g(p^{n-k})\right|\\
\ll_d & \sum_{k=1}^d \left|\alpha_k(f,p) f(p^{n-k})+\alpha_k(g,p) g(p^{n-k})\right|\\
\ll_d & \sum_{k=1}^d \left|\alpha_k(f,p)\left(f(p^{n-k})-g(p^{n-k})\right)|+|g(p^{n-k})\left(\alpha_k(f,p)-\alpha_k(g,p)\right)\right|\\
\ll_d & \sum_{k=1}^d  \left|f(p^{n-k})-g(p^{n-k})\right| \\
&+ n^{d-1}\left(|f(p)-g(p)|+ \cdots +|f(p^d)-g(p^d)|\right),
\end{align*}

Since 
\[
\sum_{n=1}^\infty \frac{|f(p^n)-g(p^n)|}{p^{n\sigma}} \ll_d \sum_{n=1}^\infty \frac{n^{d-1}}{p^{n\sigma}}
\]
is convergent, this inequality leads to
\begin{align*}
\sum_{n=d+1}^\infty \frac{|f(p^n)-g(p^n)|}{p^{n\sigma}} \ll_d &\sum_{k=1}^d  \sum_{n=d+1}^\infty\frac{|f(p^{n-k})-g(p^{n-k})|}{p^{n\sigma}} \\
&+\sum_{n=d+1}^\infty \frac{n^{d-1}}{p^{n\sigma}}\big(|f(p)-g(p)|+ \cdots +|f(p^d)-g(p^d)|\big)\\
\ll_d & \frac{1}{p^\sigma}\sum_{n=1}^\infty \frac{|f(p^n)-g(p^n)|}{p^{n\sigma}}.
\end{align*}
Therefore for all sufficiently large $p$, we have
\[
\sum_{n=d+1}^\infty \frac{|f(p^n)-g(p^n)|}{p^{n\sigma}} \ll_d  \sum_{n=1}^d \frac{|f(p^n)-g(p^n)|}{p^{n\sigma}}.
\]
By summing over $p$, we get the conclusion.
\end{proof}

\subsection{Squarefree sieving}
\label{subsec:squarefree}

Let $\chi$ be a quadratic character and let $\tilde{\chi}$ be given by
\[
\tilde{\chi}(n) := \left\{
\begin{array}{ll}
\chi(n) & \text{if $n$ is squarefree}\\
0 & \text{otherwise.}
\end{array} \right.
\]
Then Theorem \ref{thm:beta2} applies to $\chi$ and $\tilde{\chi}$ with any $\alpha>0$, $\beta>0$, and $\delta>0$, leading to 
\[
S_{\tilde{\chi}}(X) = O(X^{1/2+\epsilon})
\]
for any $\epsilon>0$. Now we remark that if $1/4<\gamma<1/2$ is fixed, the statement that
\[
S_{\tilde{\chi}}(X) = O(X^{\gamma+\epsilon})
\]
holds for all $\chi$ and for all $\epsilon>0$, is equivalent to the quasi-Riemann hypothesis for $2\gamma$, that $\zeta(s)$ has no zeros in the region $\Re(s)>2\gamma$. To see this, note that
\[
L(s,\tilde{\chi}) = L(s,\chi) \zeta(2s)^{-1}
\]
where
\[
L(s,\tilde{\chi}) = \prod_p \left(1+\frac{\chi(p)}{p^s}\right).
\]
If the quasi-Riemann hypothesis were true for $2\gamma$, then, for all $\chi$, $L(s,\tilde{\chi})$ would have no pole on $\mathrm{Re}(s)>\gamma$, and Perron's formula applies to show the result. On the other hand, if 
\[
S_{\tilde{\chi}}(X) = O(X^{\gamma+\epsilon})
\]
holds for all $\chi$ with some fixed $\gamma$, then 
\[
L(\rho,\chi)=0
\]
for any $\rho$ such that $\zeta(2\rho)=0$ and $\Re(\rho)>\gamma$. Since all such $\rho$ satisfy $\Re(\rho)<1/2$, by the prime number theorem, $\rho$ is the common zero of all $L(s,\chi)$ off the critical line. Now applying a zero density estimate, we see that there can not be any such zero, hence $\zeta(s)\neq 0$ for $1>\mathrm{Re}(s)>2\gamma$.

\bibliographystyle{abbrv}
\bibliography{power}
\end{document}